\documentclass[12pt,a4paper,reqno]{amsart}
\usepackage[utf8]{inputenc}
\usepackage[T1]{fontenc}
\usepackage{amsmath}
\usepackage{amsthm}
\usepackage{amssymb}
\usepackage[abbrev]{amsrefs}
\usepackage{mathrsfs}
\usepackage[dvipsnames]{xcolor}
\usepackage{bm}
\usepackage{enumitem}
\usepackage{hyperref}
\usepackage{bbm}

\makeatletter
\newcommand{\leqnos}{\tagsleft@true\let\veqno\@@leqno}
\newcommand{\reqnos}{\tagsleft@false\let\veqno\@@eqno}
\reqnos
\makeatother

\AtBeginDocument{\def\MR#1{}}

\newtheorem{thm}{}[section]
\newtheorem{theorem}[thm]{Theorem}
\newtheorem{corollary}[thm]{Corollary}
\newtheorem{lemma}[thm]{Lemma}
\newtheorem{proposition}[thm]{Proposition}

\theoremstyle{definition}
\newtheorem{definition}[thm]{Definition}
\theoremstyle{remark}

\newtheorem{remark}[thm]{Remark}

\numberwithin{equation}{section}
\allowdisplaybreaks

\newcommand{\SL}{\ensuremath{\mathscr{L}}}
\newcommand{\Ts}{\ensuremath{\mathcal{T}}}
\newcommand{\FF}{\ensuremath{\mathbb{F}}}

\newcommand{\NN}{\ensuremath{\mathbb{N}}}
\newcommand{\xx}{\ensuremath{\bm{x}}}
\newcommand{\XX}{\ensuremath{\mathbb{X}}}
\newcommand{\XB}{\ensuremath{\mathcal{X}}}

\newcommand{\YY}{\ensuremath{\mathbb{Y}}}

\newcommand{\EE}{\ensuremath{\mathcal{E}}}
\newcommand{\Ind}{\ensuremath{\mathbbm{1}}}
\newcommand{\Id}{\ensuremath{\mathrm{Id}}}
\newcommand{\GG}{\ensuremath{\mathcal{G}}}
\newcommand{\UU}{\ensuremath{\mathcal{R}}}
\newcommand{\BB}{\ensuremath{\mathcal{B}}}
\newcommand{\ww}{\ensuremath{\bm{w}}}
\newcommand{\ff}{\ensuremath{\bm{f}}}
\newcommand{\LO}{\ensuremath{\mathcal{L}}}
\newcommand{\Cu}{\ensuremath{\mathcal{Q}}}
\newcommand{\kk}{\ensuremath{\bm{k}}}
\newcommand{\leb}{\ensuremath{\bm{L}}}
\newcommand{\Ct}{\ensuremath{\bm{C}}}
\newcommand{\dem}{\ensuremath{\bm{\mu}}}
\newcommand{\unc}{\ensuremath{\bm{k}}}

\DeclareMathOperator{\sgn}{sign}
\DeclareMathOperator{\spn}{span}

\subjclass[2010]{41A65, 41A46, 41A17, 46B15, 46B45}

\keywords{Non-linear approximation, quasi-greedy basis, GT-space, Thresholding greedy algorithm}

\begin{document}

\title[Weak forms of unconditionality in greedy approximation]{Weaker forms of unconditionality of bases in greedy approximation}

\author[Albiac]{Fernando Albiac}
\address{Department of Mathematics, Statistics, and Computer Sciencies--InaMat2 \\
Universidad P\'ublica de Navarra\\
Campus de Arrosad\'{i}a\\
Pamplona\\
31006 Spain}
\email{fernando.albiac@unavarra.es}

\author[Ansorena]{Jos\'e L. Ansorena}
\address{Department of Mathematics and Computer Sciences\\
Universidad de La Rioja\\
Logro\~no\\
26004 Spain}
\email{joseluis.ansorena@unirioja.es}

\author[Berasategui]{Miguel Berasategui}
\address{Miguel Berasategui\\
IMAS - UBA - CONICET - Pab I, Facultad de Ciencias Exactas y Naturales \\ Universidad de Buenos Aires \\ (1428), Buenos Aires, Argentina}
\email{mberasategui@dm.uba.ar}

\author[Bern\'a]{Pablo M. Bern\'a}
\address{Pablo M. Bern\'a\\
Departamento de Matem\'atica Aplicada y Estad\'istica, Facultad de Ciencias Econ\'omicas y Empresariales, Universidad San Pablo-CEU, CEU Universities\\ Madrid, 28003 Spain.}
\email{pablo.bernalarrosa@ceu.es}

\author[Lassalle]{Silvia Lassalle}
\address{Silvia Lassalle\\
Departamento de Matem\'atica \\
Universidad de San Andr\'es, Vito Duma 284\\
(1644) Victoria, Buenos Aires, Argentina and \\
IMAS - CONICET
} \email{slassalle@udesa.edu.ar}

\begin{abstract}
In this paper we study a new class of bases, weaker than quasi-greedy bases, which retain their unconditionality properties and can provide the same optimality for the thresholding greedy algorithm. We measure how far these bases are from being unconditional and use this concept to give a new characterization of nearly unconditional bases.
\end{abstract}

\thanks{F. Albiac acknowledges the support of the Spanish Ministry for Science and Innovation under Grant PID2019-107701GB-I00 for \emph{Operators, lattices, and structure of Banach spaces}. F. Albiac and J.~L. Ansorena acknowledge the support of the Spanish Ministry for Science, Innovation, and Universities under Grant PGC2018-095366-B-I00 for \emph{An\'alisis Vectorial, Multilineal y Aproximaci\'on.} M. Berasategui and S. Lassalle were supported by ANPCyT PICT-2018-04104. P. Bern\'a by Grants PID2019-105599GB-I00 (Agencia Estatal de Investigaci\'on, Spain) and 20906/PI/18 from Fundaci\'on S\'eneca (Regi\'on de Murcia, Spain). S. Lassalle was also supported in part by CONICET PIP 0483 and PAI UdeSA 2020-2021. F. Albiac, J.~L. Ansorena and P.~M. Bern\'a would like to thank the Erwin Schr\"odinger International Institute for Mathematics and Physics, Vienna, for support and hospitality during the programme \emph{Applied Functional Analysis and High-Dimensional Approximation}, held in the Spring of 2021, where work on this paper was undertaken.}

\maketitle

\section{Introduction}\noindent
From the abstract perspective of Banach spaces, the theory of (nonlinear) greedy approximation using bases sprang from the seminal characterization of greedy bases by Konyagin and Temlyakov in 1999 as those bases that are simultaneously unconditional and democratic \cite{KoTe1999}. These two properties are, a priori, independent of each other and we find examples of unconditional bases which are not democratic and the other way around already in the very early stages of the theory (see, e.g., \cite{AlbiacKalton2016}*{Example 10.4.4}). However, the geometry of some spaces $\XX$ can make these properties intertwine, to the extent that the unconditional semi-normalized bases in $\XX$ end up being democratic (hence greedy). This is the case of unconditional bases in Hilbert spaces, and also in the spaces $\ell_{1}$ and $c_{0}$ for instance (see \cite{DSBT2012}*{Theorem 4.1}, \cite{Woj2000}*{Theorem 3} and \cite{DKK2003}*{Corollary 8.6}).

Unconditional bases are well suited to implement the Thresholding Greedy Algorithm (TGA for short) even when they are not democratic. As it happens, the unconditionality assumption can be relaxed and still obtain bases for which the TGA behaves well enough. Wojtaszczyk \cite{Woj2000} gave a precise formulation of this fact by proving that quasi-greedy bases, introduced by Konyagin and Telmyakov also in \cite{KoTe1999}, are precisely those bases for which the TGA converges. In light of this result one could argue that being quasi-greedy is the minimal requirement we must impose to a basis in order for the TGA to be a reasonable method of approximation.

However, in studying the optimality of the TGA, other forms of unconditionality have emerged which despite being weaker than quasi-greedy bases still preserve essential properties in greedy approximation. Delving deeper into these properties is of interest both from a theoretical and a practical viewpoint. On one hand, isolating the specific features of those bases makes the theory progress; on the other hand, from a more applied approach, working with these properties allows us to obtain sharper estimates for the efficiency of the TGA (see \cite{AAB2021}).

This note is motivated by the following questions, which we shall address in sections to come:
What is the relation of these new breed of bases with the already-existing ones in the theory?
Are these types of bases all-prevailing or a rare find in Banach spaces? What is the efficiency of the greedy algorithm relative to these bases? Before going into the matter we will gather the most relevant terminology in the next section.

\section{Preliminaries on terminology and background}\noindent We employ the standard notation and terminology commonly used in Banach space theory and approximation theory, as the reader will find, e.g., in the monograph \cite{AlbiacKalton2016} or the recent article \cite{AABW2021}.

Let $\XX$ be an infinite-dimensional separable Banach space (or, more generally, a quasi-Banach space) over the real or complex field $\FF$. Throughout this paper by a \emph{basis} of $\XX$ we mean a norm-bounded sequence $\XB=(\xx_n)_{n=1}^\infty$ that generates the entire space, in the sense that
\[
\overline{\spn}(\xx_n \colon n\in\NN)=\XX,
\]
and for which there is a (unique) norm-bounded sequence $\XB^*=(\xx_{n}^*)_{n=1}^\infty$ in the dual space $\XX^{\ast}$ such that $(\xx_{n}, \xx_{n}^{\ast})_{n=1}^{\infty}$ is a biorthogonal system. We will refer to the basic sequence $\XB^*$ in $\XX^*$ as to the \emph{dual basis} of $\XB$.

Given $A\subseteq\NN$ finite, $S_A=S_{A}[\XB,\XX]\colon\XX\to \XX$ will denote the \emph{coordinate projection} on $A$ with respect to the basis $\XB$,
\[
S_A(f)=\sum_{n\in A}\xx_{n}^{\ast}(f)\xx_{n}, \quad f\in\XX.
\]
For $f\in \XX$ and $m\in \NN$ we define
\[
\GG_{m}(f)=S_{A_m(f)}(f),
\]
where $A=A_m(f)\subseteq\NN$ is a \emph{greedy set} of $f$ of cardinality $m$, i.e., $|\xx_{n}^{\ast}(f)|\ge| \xx_{k}^{\ast}(f)|$ whenever $n\in A$ and $k\not\in A$. The set $A$ depends on $f$ and $m$, and may not be unique; if this happens we take any such set. Thus, the \emph{greedy operator} $\GG_{m}$ is well-defined, but is not linear nor continuous. The basis
$\XB=(\xx_{n})_{n=1}^{\infty}$ is said to be \textit{quasi-greedy} provided that there is a constant $C\ge 1$ so that for every $f\in \XX$ and for every $m\in \NN$ we have
\[
\Vert \GG_{m}(f)\Vert \le C\Vert f\Vert.
\]
Equivalently, by \cite{Woj2000}*{Theorem 1} (see also \cite{AABW2021}*{Theorem 4.1}), $\XB$ is quasi-greedy if and only if
\[
\lim_{m\to\infty} \GG_{m}(f)=f\; \text{for each}\; f\in \XX,
\]
i.e., the TGA $(\GG_{m})_{m=1}^{\infty}$ always converges.
Of course, unconditional bases are quasi-greedy, but the converse does not hold in general. Konyagin and Telmyakov provided in \cite{KoTe1999} the first examples of \emph{conditional} (i.e., not unconditional) quasi-greedy bases.

We also recall that a basis $\XB=(\xx_n)_{n=1}^\infty$ is called \emph{democratic} if there is a constant $\Delta\ge 1$ such that for any finite subsets $A$ and $B$ of $\NN$ with $|A|\le |B|$, we have
\[
\left\Vert \sum_{k\in A}\xx_k \right\Vert\le \Delta \left\Vert \sum_{k\in B}\xx_k \right\Vert.
\]
The \emph{upper democracy function}, also called the \emph{fundamental function}, of $\XB$ is then defined by
\[
\varphi(m)=\varphi[\XX, \XB](m)=\sup\limits_{|A|\le m}\left\Vert \sum_{k\in A}\xx_k \right\Vert.
\]
Given $A\subseteq \NN$, we will use $\EE_A$ to denote the set consisting of all families $\varepsilon=(\varepsilon_n)_{n\in A}$ in $\FF$ with $|\varepsilon_n|=1$ for all $n\in A$, and will put
\[
\Ind_{\varepsilon,A}=\Ind_{\varepsilon,A}[\XB,\XX]=\sum_{n\in A} \varepsilon_n\, \xx_n.
\]
If the basis and the space are clear from the context we simply put $\Ind_{\varepsilon,A}$.

In order to better understand how certain greedy-like properties dualize, Dilworth et al.\ introduced in \cite{DKKT2003} a strengthend form of democracy. They called a basis $\XB$ \emph{bidemocratic} provided there is a constant $C$ such that
\[
\varphi[\XB,\XX](m) \, \varphi[\XB^*,\XX^*](m)\le C m, \quad m\in\NN.
\]
Bidemocratic bases are, in particular democratic \cite{DKKT2003}*{Proposition 4.2}.

Another property linked to democracy which will be of interest for us is squeeze symmetry. We say that a basis $\XB$ is \emph{squeeze symmetric} if there are symmetric bases $\BB_1$ and $\BB_2$ in respective quasi-Banach spaces $\YY_1$ and $\YY_2$ such that $\BB_1$ dominates $\XB$, $\XB$ dominates $\BB_2$, and
\[
\varphi[\BB_1,\YY_1](m) \lesssim \varphi[\BB_2,\YY_2](m), \quad m\in\NN.
\]
For instance if $\XX$ is locally convex (i.e., $\XX$ is a Banach space) and $\XB$ dominates the unit vector system of the space weak-$\ell_1$, then $\XB$ is squeeze symmetric. Squeeze symmetry is an intermediate property between democracy and bidemocracy and is considered by some authors as a condition which ensures in a certain sense the optimality of the compression algorithms with respect to the basis (see \cite{Donoho1993}). We refer the reader to \cite{AABW2021}*{\S9} for more details.

Working with the TGA leads naturally to consider derived forms of unconditionality which are still of interest in greedy approximation using bases. For instance, in the early days of the theory Wojtaszczyk proved that quasi-greedy bases in quasi-Banach spaces are suppression unconditional for constant coefficients, or SUCC for short (\cite{Woj2000}*{Proposition 2}). This means that there is a constant $C$ such that
\begin{equation}\label{succ}
\Vert \Ind_{\varepsilon,B}\Vert \le C \Vert \Ind_{\varepsilon,A}\Vert\end{equation} for every $A\subseteq\NN$ finite, every $\varepsilon\in\EE_A$, and every $B\subseteq A$. The smallest constant $C\ge 1$ so that \eqref{succ} holds will be called the SUCC constant of $\XB$.
As the subject evolved, new attributes of bases as well as new nonlinear operators associated with them were identified, and the ingredients that played a key role in the foundational results of the theory were given the status they deserved. We single out a trait of bases that will be the focus of our attention in what follows.

Let $\XB=(\xx_{n})_{n=1}^{\infty}$ be a basis of a quasi-Banach space $\XX$. For $f\in \XX$ and $A\subseteq\NN$ finite, put
\[
\UU(f,A) = \min_{n\in A} |\xx_n^*(f)| \sum_{n\in A} \sgn(\xx_n^*(f)) \, \xx_n.
\]
For $m\in\NN\cup\{0\}$, the $m$\emph{th-restricted truncation operator} $\UU_m\colon \XX \to \XX$ is defined as
\[
\UU_m(f)=\UU(f,A_m(f)), \quad f\in\XX.
\]
The truncation operators
\[
\Ts_m=\UU_m+\Id_{\XX}-\GG_m,\quad m\in \NN,
\]
were introduced in \cite{DKK2003} to give relief to the core feature of quasi-greedy bases in the proof that quasi-greedy democratic bases are almost greedy (see \cite{DKKT2003}*{Theorem 3.3}). Further developments in the theory have shown that the `restricted' component of the operators $(\Ts_m)_{m=1}^{\infty}$ yields more accuracy when the TGA is implemented for non-quasi-greedy bases. Thus, the uniform boundedness of the restricted truncation operators was singled out as a property of interest by itself in greedy approximation (see \cite{BDKOW2019}*{Definition 3.12}).

The authors of \cite{AABW2021} used the uniform boundedness of the restricted truncation operators to extend to the nonlocally convex setting the characterization of almost greedy bases. It is worth pointing out that this extension is far from trivial (see \cite{AABW2021}*{Theorem 6.3}). In this paper, we will call bases with
\[
\sup_{m\in\NN} \Vert \UU_m\Vert<\infty
\]
\emph{truncation quasi-greedy}.

In the same way as unconditionality is an ingredient of greediness and quasi-greediness is an ingredient of almost greediness, truncation quasi-greediness is an ingredient of squeeze symmetry. In fact, a basis is squeeze symmetric if and only if it is truncation quasi-greedy and democratic (see \cite{AABW2021}*{Lemma 9.3 and Theorem 9.12}).

If $\XB$ is quasi-greedy, then $\XB$ is truncation quasi-greedy and the restricted truncation operators converge, i.e., $\lim_{m\to \infty} \UU_m(f)=0$ for all $f\in \XX$ (see \cite{DKKT2003}*{Lemma 2.2} and \cite{AABW2021}*{Theorem 4.13}). In practice, in most situations the only property of quasi-greedy bases that one needs is that they are truncation quasi-greedy (see e.g.\ \cite{AABBL2021}*{Theorem 2.4} or \cite{AABW2021}*{Proposition 10.17}). However, these two concepts are not the same. The first example that illustrates this dissimilitude can be found in the proof of \cite{BBG2017}*{Proposition 5.6}, where the authors constructed a basis that dominates the unit vector system of weak-$\ell_1$ hence it is truncation quasi-greedy, but it is not quasi-greedy. In this regard, it must be mentioned that it was recently proved that bidemocratic bases need not be quasi-greedy \cite{AABBL2021}*{Corollary 3.7}. Since bidemocratic bases are truncation quasi-greedy, this result yields, in particular, the existence of truncation quasi-greedy bases that are not quasi-greedy.

\section{How far are truncation quasi-greedy bases from being unconditional?}\label{sect:Ntruncation quasi-greedy}\noindent
In spite of the fact that truncation quasi-greedy bases need not be quasi-greedy, they still enjoy most of the nice unconditionality-like properties of quasi-greedy bases. For instance, they are quasi-greedy for large coefficients, suppression unconditional for constant coefficients, and lattice partially unconditional. See \cite{AABW2021}*{Sections 3 and 4} for the precise definitions and the proofs of these relations.

This section is devoted to providing an answer to the question in the title, both from a qualitative and a quantitative point of view. The qualitative approach will consist in identifying truncation quasi-greedy bases with an already existing class of bases in greedy approximation, namely, nearly unconditional bases. To tackle the quantitative approach, we will pay attention to the growth of the unconditionality constants associated with the basis. For broader applicability of our results we work in the general framework of quasi-Banach spaces.

Let us recall a notion introduced by Elton in \cite{Elton1978} and imported to greedy approximation theory by Dilworth et al.\ in \cite{DKK2003}. Given a quasi-Banach space $\XX$ with a basis $\XB=(\xx_n)_{n=1}^\infty$ we consider the set $\Cu$ of vectors in $\XX$ whose coefficient sequences (relative to the basis $\XB$) belong to the unit ball of $\ell_{\infty}$,
\[
\Cu=\mathcal Q[\XB,\XX]=\{f\in \XX \colon \sup_n|\xx_n^*(f)| \le 1\}.
\]
Now, given a threshold number $a\ge 0$ and $f\in \XX$ put
\[
A(a,f):=\{n\in\NN \colon |\xx_n^*(f)|\ge a\}.
\]
The basis $\XB$ is said to be \emph{nearly unconditional} if for each $a\in(0,1)$ there is a constant $C$ such that for any $f\in \Cu$,
\begin{equation}\label{nearlyuncdef}
\Vert S_A(f)\Vert\le C\Vert f\Vert,
\end{equation}
whenever $A\subseteq A(a,f)$.

Let $\phi\colon(0,1)\to(0,\infty)$ denote the function that maps each $a$ to the the smallest constant $C$ in \eqref{nearlyuncdef}. If the inequality in \eqref{nearlyuncdef} holds only for $A =A(a,f)$, the basis is said to be \emph{thresholding-bounded}, and we denote by $\theta(a)$, $0<a<1$ the smallest constant $C$. Notice that $\phi$ is bounded if and only if $\XB$ is unconditional, and that the function $\theta$ is bounded if and only if the basis is quasi-greedy.

Nearly unconditional bases are thresholding-bounded in the same way that unconditional bases are quasi-greedy. However, while quasi-greedy bases need not be unconditional, thresholding-bounded bases are surprisingly nearly unconditional (see \cite{DKK2003}*{Proposition 4.5}). Our first goal in this section will be to show that when truncation quasi-greedy bases come into play, the corresponding characterization also holds.

If $A=A(a,f)$, we will use the thresholding operators
\[
\GG^{(a)}(f):=S_A(f)\quad \text{and}\quad \UU^{(a)}(f):=\UU_A(f).
\]

\begin{definition}
A basis $\XB$ of a quasi-Banach space $\XX$ will be said to be \emph{nearly truncation quasi-greedy} if for each $a\in(0,1)$ there is a constant $C$ such that
\begin{equation}\label{nearlytrucqg}
\Vert \UU^{(a)}(f)\Vert \le C \Vert f\Vert, \quad f\in \mathcal Q.
\end{equation}
We will denote by $\lambda(a)$ the smallest constant $C$ so that \eqref{nearlytrucqg} holds. The mapping $\lambda\colon(0,1) \to (0,\infty)$ will be called the nearly truncation quasi-greedy function of $\XB$.
\end{definition}

\begin{lemma}\label{lem:Ntruncation quasi-greedy}
Let $\XB$ be a nearly truncation quasi-greedy basis of a quasi-Banach space $\XX$. Then:
\begin{enumerate}[label=(\roman*), leftmargin=*, widest=ii]
\item\label{it:A} Its nearly truncation quasi-greedy function $\lambda$ is non-increasing, and
\item\label{it:B}
$\XB$ is SUCC with constant $\lambda(1^-)$.
\end{enumerate}
\end{lemma}
\begin{proof}If $0<a\le b <1$ and $f\in \Cu$, then $a f/b\in \Cu$, and
\[
\UU^{(b)} (f) = \frac{b}{a} \UU^{(a)} \left( \frac{a}{b} f \right).
\]
This gives \ref{it:A}. To see \ref{it:B}, we note that a standard perturbation technique yields $\Vert \Ind_{\varepsilon,A} \Vert \le \lambda(1^-) \Vert \Ind_{\varepsilon,B} \Vert $ whenever $A\subseteq B$.
\end{proof}

\begin{lemma}\label{lem:NRPLPU}
Suppose that $\XB=(\xx_n)_{n=1}^\infty$ is a nearly truncation quasi-greedy basis of a quasi-Banach space $\XX$. Then there is a constant $C$ depending only on the modulus of concavity of the quasi-norm on $\XX$ such that
\[
\left\Vert \sum_{n\in A} a_n\, \xx_n\right\Vert \le C \lambda(1^-)\lambda(a) \Vert f\Vert, \quad f\in \Cu,
\]
for all $ 0<a<1$ and all $A\subseteq\NN$ such that $|a_n|\le a \le |\xx_n^*(f)|$ for all $n\in A$.
\end{lemma}

\begin{proof}
By Lemma~\ref{lem:Ntruncation quasi-greedy}~\ref{it:B} and \cite{AABW2021}*{Lemma 3.2} there is a constant $C$ such that
\[
\left\Vert \sum_{n\in A} a_n\, \xx_n\right\Vert \le C \lambda(1^-) \max_{n\in A} |a_n| \, \Vert \Ind_{\varepsilon, A} \Vert,
\; A\subseteq\NN, \, (a_n)_{n\in A}\in\FF^A,\, \varepsilon\in\EE_A.
\]
Combining the definition of nearly truncation quasi-greedy basis with this inequality yields the desired result.
\end{proof}

\begin{theorem}\label{thm:NUNtruncation quasi-greedy}
Let $\XB=(\xx_n)_{n=1}^\infty$ be a basis of a quasi-Banach space $\XX$. Then $\XB$ is nearly unconditional if and only if it is nearly truncation quasi-greedy. Moreover,
\[
\lambda(a)\lesssim \theta(a)\lesssim \phi(a) \lesssim\frac{1}{a} \lambda(a), \quad 0<a<1.
\]
\end{theorem}

\begin{proof}
Without lost of generality we may assume that $\XX$ is $p$-Banach for some $0<p\le 1$. Suppose that $\XB$ is thresholding bounded. Then it is SUCC and so by \cite{AABW2021}*{Lemma 3.6} there are constants $s>1$ and $C_s>1$ such that for any finite set $A\subseteq \NN$ and any $\varepsilon \in \EE_A$,
\[
\Vert \Ind_{\varepsilon,A}\Vert \le C_s\left\Vert \sum_{n\in A} a_n\, \xx_n\right\Vert, \qquad 1\le |a_n|\le s.
\]
For $f\in \mathcal Q\setminus \{0\}$ put
\[
A_1=\{n\in\NN \colon s^{-1} \le |\xx_n^*(f)| \leq 1\},
\]
and
\[
A_j=\{n\in\NN \colon s^{-j} \le |\xx_n^*(f)| < s^{-j+1}\}, \quad j\in\NN\setminus\lbrace 1\rbrace.
\]
Fix $k$ such that $\UU^{(s^{-k})}(f)\not=0$, and let
\[
l=:\max\{1\le j\le k : A_j\not=\emptyset \}.
\]
Then
\[
\UU^{(s^{-k})}(f)=\UU^{(s^{-l})}(f)\]
\text{and}
\[ s^{-l}\le b:= \min_{n\in A_l}|\xx_n^*(f)|<s^{-l+1}.
\]
We have
\[
\UU^{(s^{-l})}(f) = s^{l}b \sum_{j=1}^l s^{-l+j} s^{-j} \Ind_{\varepsilon(f),A_j}.
\]
Therefore, by $p$-convexity,
\begin{align*}
\|\UU^{(s^{-l})}(f)\|
&\le s^{l}b \left( \sum_{j=1}^l s^{(j-l)p}\right)^{1/p} \sup_{1\le j \le l} s^{-j} \Vert \Ind_{\varepsilon,A_j} \Vert\\
&\le C_{p,s} \sup_{1\le j \le l} s^{-j} \Vert \Ind_{\varepsilon,A_j} \Vert,
\end{align*}
where $C_{p,s}=s(1-s^{-p})^{-1/p}$. In particular, for $j=1$,
\begin{align*}
s^{-1}\|\Ind_{\varepsilon,A_1}\|& \le s^{-1}C_s \left\Vert \sum_{n\in A_1} s \xx_n^*(f) \, \xx_n\right\Vert\\
&=C_s\|\GG^{\left(s^{-1}\right)}\left(f\right)\|\\
&\le C_s\theta(s^{-1})\|f\|\\
&\le 2C_s\theta(s^{-1})\theta\left(s^{-k+1}\right)\|f\|.
\end{align*}
If $l>1$, for $2\le j \le l$ we have
\begin{align*}
s^{-j} \| \Ind_{\varepsilon,A_j} \|&\le s^{-j}C_s\left\Vert \sum_{n\in A_j} s^j \xx_n^*(f) \, \xx_n\right\Vert\\
&=C_s\left\Vert \GG^{\left(s^{-j}\right)}\left(f-\GG^{\left(s^{-j+1}\right)}\left(f\right)\right)\right\Vert\\
&=C_s\left\Vert s^{-j+1} \GG^{\left(s^{-1}\right)}\left(s^{j-1}\left(f-\GG^{\left(s^{-j+1}\right)}\left(f\right)\right)\right)\right\Vert\\
&\le C_s s^{-j+1}\theta\left(s^{-1}\right)\left\Vert s^{j-1}\left(f-\GG^{\left(s^{-j+1}\right)}\left(f\right)\right)\right\Vert\\
&\le 2 C_s\theta\left(s^{-1}\right)\theta \left(s^{-j+1}\right)\|f\|\\
& \le 2 C_s\theta\left(s^{-1}\right)\theta \left(s^{-k+1}\right)\|f\|.
\end{align*}
This proves that
\begin{align*}
\lambda\left(s^{-k}\right)\le& 2C_{p,s} C_s \theta\left(s^{-1}\right)\theta \left(s^{-k+1}\right), \qquad k\in \NN.
\end{align*}

Given $0<a<1$ and $f$ as before, pick $k\in \NN$ so that $s^{-k}\le a<s^{-k+1}$. Since $\theta$ and $\lambda$ are both non-increasing,
\begin{align*}
\lambda\left(a\right)\le& \lambda\left(s^{-k}\right)\le 2C_{p,s} C_s \theta\left(s^{-1}\right)\theta \left(s^{-k+1}\right)\le 2C_{p,s} C_s \theta\left(s^{-1}\right)\theta \left(a\right).
\end{align*}
Suppose that $\XB$ is nearly truncation quasi-greedy. Then it is SUCC by Lemma~\ref{lem:Ntruncation quasi-greedy}(ii). Therefore, by \cite{AABW2021}*{Lemma 3.2}, there is $C_u>1$ such that
\[
\left\Vert \sum_{n\in A} a_n\, \xx_n\right\Vert \le C_u \Vert \Ind_{\varepsilon,A}\Vert,
\quad |a_n|\le 1, \ |A|<\infty.
\]

Let $a\le 1/2$. Pick $s\in[2,4)$ and $k\in\NN$ such that $a=s^{-k}$. Let $f\in \Cu$, $k\in\NN$ and $A\subseteq \{n\in\NN \colon |\xx_n^*(f)|\ge s^{-k}\}$. Set
\begin{align*}
A_j&=\{n\in\NN \colon s^{-j}\le |\xx_n^*(f)| < s^{-j+1}\}, \quad j\in\NN,\\
B_j&=\{n\in\NN \colon s^{-j}\le |\xx_n^*(f)| \}, \quad j\in\NN\cup\{0\}
\end{align*}
(in the definition of $A_1$ we replace `$<$' with `$\le$'). Reasoning as before from the expansion
\begin{equation}
S_A(f)= \sum_{j=1}^k s^{-j} s^j S_{A_j}(f) \label{ec3}
\end{equation}
and taking into account that $\lambda$ is non-increasing, we obtain
\begin{align*}
\Vert S_A(f)\Vert
&\le C_{p,s} C_u \sup_{1\le j \le k} s \Vert \Ind_{\varepsilon(f),A_j} \Vert \\
&\le C_{p,s} C_u 2^{1/p} s \sup_{1\le j \le k} ( \Vert \Ind_{\varepsilon(f),B_j} \Vert+ \Vert \Ind_{\varepsilon(f),B_{j-1}}\Vert)\\
&\le C_{p,s} C_u 2^{1/p+1} s \sup_{1\le j \le k} (s^{j-1} \lambda(s^{-j+1}) + s^j \lambda(s^{-j}))\\
&\le C_{p,s} C_u 2^{1/p+1} (s+1) \frac{\lambda(a)}{a} \\
&\le \frac{5 2^{1/p+5}}{(2^{p}-1)^{1/p}} C_u \frac{\lambda(a)}{a}.
\qedhere
\end{align*}
\end{proof}

Our next result is a straightforward consequence of Theorem~\ref{thm:NUNtruncation quasi-greedy}.

\begin{theorem}
Truncation quasi-greedy bases are nearly unconditional.
\end{theorem}

In order to quantify the conditionality of a basis $\XB=(\xx_n)_{n=1}^\infty$ in a quasi-Banach space $\XX$, it is customary to study the growth of its \emph{unconditionality constants}
\[
\kk_m=\kk_m[\XB,\XX] :=\sup_{|A|\le m} \Vert S_A[\XB,\XX]\Vert, \quad m\in\NN.
\]

An asymptotic upper bound for the unconditionality constants of truncation quasi-greedy bases in $p$-Banach spaces was estimated in \cite{AAW2021b}, where the following theorem was proved.

\begin{theorem}[\cite{AAW2021b}*{Theorem 5.1}]\label{thm:estimatetruncqg}
Let $\XB$ be a truncation quasi-greedy basis of a $p$-Banach space $\XX$, $0<p\le 1$. Then
\[
\kk_m[\XB,\XX]\lesssim (\log m)^{1/p} , \quad m\ge 2.
\]
\end{theorem}

Next we show that it is also possible to quantify the conditionality of nearly unconditional basis in terms of the growth of the function $\lambda$, thus extending \cite{DKK2003}*{Lemma 8.2} and Theorem~\ref{thm:estimatetruncqg}.

\begin{theorem}\label{thm:NUCC}
Let $\XX$ be a $p$-Banach space, $0<p\le 1$. If $\XB=(\xx_n)_{n=1}^\infty$ is a nearly unconditional basis of $\XX$ then
\[
\kk_m[\XB,\XX] \lesssim \lambda\left( \frac{1}{m^{1/p}} \right) (\log m)^{1/p}, \quad m\ge 2.
\]
\end{theorem}

\begin{proof}
Let $N\in\NN$ with $2^N \le m <2^{N+1}$. Let $A\subseteq\NN$ with $|A|\le m$ and let $f\in\XX$ with $\Vert f \Vert \le 1/c$, where $c=\sup_{n \in\NN} \Vert \xx_n^*\Vert$. Then $f\in \Cu$. Set
\[
A_0=\{n\in A\colon 2^{-N/p}< |\xx_n^*(f)|\}
\]
and $A_1=A\setminus A_0$. We have $S_A(f)=S_{A_0}(f)+S_{A_1}(f)$. Denote $d=\sup_{n\in\NN} \Vert \xx_n\Vert$. Since $|A_1|\le m$,
\[
\Vert S_{A_1} (f)\Vert^p \le |A_1| d^p 2^{-N}\le 2 d^p.
\]
Let $(B_k)_{k=1}^N$ be the partition of $A_0$ given by
\[
B_k=\{n\in A \colon 2^{-k/p}< |\xx_n^*(f)|\le 2^{(-k+1)/p}\},
\]
so that $S_{A_0}(f)=\sum_{k=1}^N S_{B_k}(f)$. We have
\[
\max_{n\in B_k} |\xx_n^*(2^{-1/p}f)|\le 2^{-k/p} \le \min_{n\in B_k} |\xx_n^*(f)|, \quad k=1,\dots, N.
\]
An application of Lemma~\ref{lem:NRPLPU} gives
\[
2^{-1/p}\Vert S_{B_k} (f)\Vert \le C \lambda(1^-) \lambda( 2^{-k/p})\Vert f\Vert, \quad k=1,\dots, N.
\]
Finally, since $\lambda$ is non-increasing,
\begin{align*}
\Vert S_{A_0}(f)\Vert^p
&\le 2 C^p \lambda^p(1^-) \sum_{k=1}^N \lambda^p( 2^{-k/p}) \Vert f\Vert^p \\
&\le 2 \frac{1}{c^p} C^p \lambda^p(1^-) \log_2(m) \lambda^p( m^{-1/p}).\qedhere
\end{align*}
\end{proof}

\begin{remark}
The best one can say about the asymptotic estimates for the unconditionality constants of general quasi-greedy bases in Banach spaces is that
\[
\kk_{m} \lesssim \log m, \qquad m\ge 2.
\]
This is the statement of Theorem~\ref{thm:estimatetruncqg} for $p=1$, which in this particular case follows by a result of Dilworth et al.\ (see \cite{DKK2003}*{Lemma 8.2}). Later on, it was evinced that the geometry of certain Banach spaces may contribute to improve this estimate. For instance, quasi-greedy bases in super-reflexive Banach spaces verify instead
\begin{equation}\label{AAGHRestimate}
\kk_{m}[\XB,\XX] \lesssim (\log m)^{1-\varepsilon}, \qquad m\ge 2.
\end{equation}
for some $0<\varepsilon<1$ depending on $\XB$ and $\XX$ (see \cite{AAGHR2015}*{Theorem 1.1}). As a possible limitation of our methods in this paper, we notice that they do not allow one to extend \eqref{AAGHRestimate} to truncation quasi-greedy bases. This suggests the question of whether truncation quasi-greedy bases in super-reflexive Banach space $\XX$ are susceptible to better asymptotic estimates along the lines of \cite{AAGHR2015}*{Theorem 1.1}. \end{remark}

\section{Existence and uniqueness of truncation quasi-greedy bases in Banach spaces}\noindent
The results of this section aim at complementing and extending to truncation quasi-greedy bases the theoretical study on the existence and uniqueness of quasi-greedy bases in Banach spaces from \cite{DKK2003}. In this pioneering article, Dilworth et al.\ proved that $c_{0}$ is the unique Banach space (up to isomorphism) with a unique quasi-greedy basis (up to equivalence). Their results relied on very deep concepts from classical Banach space theory, so we begin by recalling these important facts. A Banach space $\XX$ is called a \emph{GT space} (after ``Grothendieck Theorem'') \cite{Pisier1986} if every bounded linear operator $T\colon \XX\to\ell_2$ is absolutely summing, i.e., there is a constant $C$ such that for all finite collections of functions $(f_k)_{k\in B}$ in $\XX$,
\begin{equation}\label{ASE}
\sum_{k\in B} \Vert T(f_k) \Vert_2
\le C \sup \left\{ \left\Vert \sum_{k\in B} \varepsilon_k f_k\right\Vert \colon (\varepsilon_k)_{k\in B}\in\EE_B \right\}.
\end{equation}
The smallest constant $C$ in \eqref{ASE} is the absolutely summing norm of $T$ and is denoted by $\pi_1(T)$. By the Closed Graph theorem, if $\XX$ is a GT space, there is a constant $C_g$, called the GT constant of $\XX$, such that $\pi_1(T)\le C_g \Vert T \Vert$ for all $T\in\LO(\XX,\ell_2)$.

For instance, Lindenstrauss and Pe{\l}czy\'{n}ski \cite{LinPel1968} proved that $L_{1}(\mu)$ spaces and, more generally, $\SL_1$-spaces are GT spaces. In turn, if $\XX$ is an $\SL_{\infty}$-space then $\XX^{\ast}$ is an $\SL_1$-space and, then, a GT space (see \cite{LinRos1969}). The aforementioned result from \cite{DKK2003} on the existence and uniqueness of quasi-greedy bases in Banach spaces relies in part on the following proposition which we will also use.

\begin{proposition}[\cite{DKK2003}*{Proposition 8.1}]\label{DKK2003Prop8.1}
Suppose that $\XB$ is a thresholding bounded Hilbertian basis of a Banach space $\XX$. If $\XX^{\ast}$ is a GT space, then $\XB$ is equivalent to the unit vector basis of $c_{0}$.
\end{proposition}

For expository sake we recall that a basis $\XB=(\xx_{n})_{n=1}^{\infty}$ of a quasi-Banach space $\XX$ is \emph{$p$-Hilbertian}, $0<p\le \infty$, if it is dominated by the unit vector basis of $\ell_p$, i.e., $(\xx_{n})_{n=1}^{\infty}$ satisfies the upper $p$-estimate,
\begin{equation}\label{pHilbertian}
\Vert f\Vert= \left\Vert \sum_{n=1}^{\infty}\xx_{n}^{\ast}(f)\xx_n\right\Vert\le C_{p}\left(\sum_{n=1}^{\infty}|\xx_{n}^{\ast}(f)|^{p}\right)^{1/p}
\end{equation}
for some constant $C_{p}$. The basis is called \emph{Hilbertian} if it is $2$-Hilbertian. A $p$-Hilbertian basis induces a continuous linear embedding of $\ell_{p}$ into $\XX$,
\[
\ell_{p}\hookrightarrow \XX, \quad (a_{n})_{n=1}^{\infty}\mapsto \sum_{n=1}^{\infty} a_{n}\xx_{n},
\]
whose norm is the smallest constant $C_{p}$ in \eqref{pHilbertian}.

Of course, any basis of a $p$-Banach space, $0<p\le 1$, is $p$-Hilbertian. If $\XB=(\xx_n)_{n=1}^\infty$ is $p$-Hilbertian for some $0< p<\infty$, by duality the map
\[
\XX^*\to \FF^\NN, \quad f^*\mapsto (f^*(\xx_n))_{n=1}^{\infty},
\]
is bounded from $\XX^*$ into $\ell_{p'}$, where $p'\in[1,\infty]$ is related to $p$ by
\[
\frac{1}{p'}=1-\frac{1}{\max\{1,p\}}.
\]

In order to obtain properties that are more closely tied to a given basis in a certain space we need to refine the embeddings, and this is accomplished using Lorentz sequence spaces. For the pupose of this paper, it suffices to consider classical Lorentz spaces. Given $0<p<\infty$ and $0<q\le\infty$, the Lorentz sequence space $\ell_{p,q}$ is the quasi-Banach space consisting of all $f\in c_0$ whose non-increasing rearrangement $(a_n)_{n=1}^\infty$ satisfies
\begin{equation}\label{eq:CLSS}
\Vert f \Vert_{\ell_{p,q}}=\left( \sum_{n=1}^\infty a_n^q n^{q/p-1}\right)^{1/q}.
\end{equation}
If $q=\infty$ the quasi-norm is defined with the usual modification in \eqref{eq:CLSS}, and the resulting space is known as weak-$\ell_p$.
We will need the following consequence of the embeddings obtained in \cite{AABW2021}.

\begin{theorem}\label{thm:AABW}
Let $\XB$ be a basis of a $p$-Banach space $\XX$, $0<p\le 1$. Suppose that, for some $0<r<\infty$, there is a constant $C$ such that
\[
\Vert \Ind_{\varepsilon,A} \Vert \le C m^{1/r}, \quad A\subseteq\NN, \, |A|\le m, \, \varepsilon\in\EE_A.
\]
Then the unit vector system of $\ell_{r,p}$ dominates $\XB$.
\end{theorem}

\begin{proof}
In agreement with the terminology in \cite{AABW2021}, $\ell_{r,p}=d_{1,p}(\ww)$, where $\ww$ is the weight with primitive sequence $(m^{1/r})_{m=1}^\infty$. Thus, the result follows from \cite{AABW2021}*{Corollary 9.13}.\end{proof}

\begin{lemma}[cf.\ \cite{DKK2003}*{Lemma 8.4}]\label{lem:BS}
Let $\XB$ be a basis of a Banach space $\XX$ such that $\kk_m[\XB,\XX]\lesssim m^{t}$ for some $0<t<1/2$. Suppose that $\XX^*$ is a GT space. Then $\XB$ is $p$-Hilbertian for all $p<1/t$.
\end{lemma}

\begin{proof}
For each $A\subseteq\NN$ finite and $\varepsilon=(\varepsilon_n)_{n\in A}\in\EE_A$, we consider the multiplier operator
\begin{equation*}
M_\varepsilon\colon\XX\to \XX,
\quad \sum_{n=1}^\infty a_n\, \xx_n \mapsto \sum_{n\in A}^\infty\varepsilon_n\, a_n \, \xx_n.
\end{equation*}
By \cite{AABW2021}*{Corollary 2.4}, there is a constant $C_t$ such that
\[
\Vert M_\varepsilon \Vert \le C_t |A|^t, \quad A\subseteq\NN, \, \varepsilon\in\EE_A.
\]
We will use a bootstrap argument: we will prove that if $t<u<1/2$ and $\XB$ is $p$-Hilbertian for some $1\le p<2$ then $\XB$ is $q$-Hilbertian, where $1/q=1/p-1/2+u$.

Let $C_p<\infty$ be the norm of the embedding of $\ell_p$ into $\XX$ via $\XB$. Pick $s\in (1, 2]$ so that $1/s=3/2-1/p$. Solving this equation gives $s=2p/(3p-2)$, $s-1=(2-p)/(3p-2)$, and $1-1/s=1/p-1/2$. Given $h^*\in\XX^*$ and $A\subseteq\NN$ finite we set $m=|A|$ and define
\[
T=T_{h^*,A} \colon \XX^* \to \ell_2(A),\quad
f^*\mapsto ( |h^*(\xx_n)|^{s-1} f^*(\xx_n) )_{n\in A}.
\]
Since $(p'/2)'=p/(2-p)$, H\"older's inequality gives
\begin{align*}
\Vert T(f^*)\Vert_2&=\left( \sum_{n\in A} |h^*(\xx_n)|^{2(s-1)} |f^*(\xx_n)|^2 \right)^{1/2}\\
&\le \left( \sum_{n\in A} |h^*(\xx_n)|^s \right)^{1-1/s} \left( \sum_{n\in A} |f^*(\xx_n)|^{p'} \right)^{1/p'}\\
&\le C_p\left( \sum_{n\in A} |h^*(\xx_n)|^s \right)^{1-1/s} \Vert f^*\Vert.
\end{align*}
Therefore if $C_g$ is the GT constant of $\XX^*$,
\begin{align}
\sum_{j\in B} \Vert T(f_j^*)\Vert \le C_p C_g \left( \sum_{n\in A} |h^*(\xx_n)|^s \right)^{1-1/s} \sup_{\varepsilon=\pm 1}
\left\Vert \sum_{j\in B} \varepsilon_j \, f_j^* \right\Vert,\label{uno}
\end{align}
for any finite family $\ff^*=(f_j^*)_{j\in B}$ in $\XX^*$. Choosing $\ff^*=(h^*(\xx_n) \xx_n^*)_{n\in A}$, we obtain
\begin{align}
\sum_{n\in A} |h^*(\xx_n)|^s &\le C_p C_g \left( \sum_{n\in A} |h^*(\xx_n)|^s \right)^{1-1/s} \sup_{\varepsilon\in\EE_A}
\Vert M_\varepsilon^*(h^*)\Vert\label{dos}\\
&\le C_p C_t C_g \left( \sum_{n\in A} |h^*(\xx_n)|^s \right)^{1-1/s}
m^{t} \Vert h^*\Vert\nonumber
\end{align}
for all $h^*\in\XX^*$, whence
\[
\left( \sum_{n\in A} |h^*(\xx_n)|^s \right)^{1/s}\le C_p C_t C_g m^{t} \Vert h^*\Vert.
\]
Let $\varepsilon=(\varepsilon_n)_{n\in A}\in\EE_A$. By the Hahn--Banach theorem and H\"older's inequality,
\begin{align*}
\Vert \Ind_{\varepsilon,A}[\XB,\XX]\Vert
&=\sup_{\Vert h^*\Vert \le 1} \left|\sum_{n\in A} \varepsilon_n \, h^*(\xx_n)\right|\\
&\le m^{1-1/s} \sup_{\Vert h^*\Vert \le 1} \left( \sum_{n\in A} |h^*(\xx_n)|^s \right)^{1/s}\\
&\le m^{1-1/s} C_p C_t C_g m^{t}= C_p C_t C_g m^{1/p-1/2+t}.
\end{align*}
Hence, by Theorem~\ref{thm:AABW}, the unit vector basis of $\ell_{r,1}$ dominates $\XB$, where $r$ is determined by $1/r=1/p-1/2+t$. We conclude by noticing that since $1/r<1/q$, the space $\ell_q$ continuously embeds into $\ell_{r,1}$.
\end{proof}

Now, we are in a position to prove the main result of this section.
\begin{theorem}\label{thm:NTQGGTDual} Let $\XB$ be a nearly unconditional basis of a Banach space. Suppose there is $0\le t<1/2$ such that the nearly truncation quasi-greedy function of $\XB$ satisfies $\lambda(a) \lesssim a^{-t}$ for $0<a<1$. If $\XX^*$ is a GT space, then $\XB$ is equivalent to the unit vector basis of $c_0$.
\end{theorem}

\begin{proof}
Combining Theorem~\ref{thm:NUCC} with Lemma~\ref{lem:BS} gives that $\XB$ is Hilbertian. Then the result follows from Proposition~\ref{DKK2003Prop8.1}.
\end{proof}
The following consequence of Theorem~\ref{thm:NTQGGTDual} improves \cite{AABBL2021}*{Theorem 3.11(ii)}, which reaches the same conclusion but under the stronger assumption that the basis is bidemocratic.

\begin{corollary} \label{thm:UBTOco}
Let $\XB$ be a truncation quasi-greedy basis of a Banach space $\XX$. If $\XX^{\ast}$ is a GT space, then $\XB$ is equivalent to the unit vector basis of $c_0$.
\end{corollary}

\begin{proof}
Just apply Theorem~\ref{thm:NTQGGTDual} with $t=0$.
\end{proof}

\section{Concluding remarks on the efficiency of the TGA relative to truncation quasi-greedy bases}\noindent

We finish our discussion on truncation quasi-greedy bases with some remarks on optimality. Our conclusion will be that, in spite of the fact that truncation quasi-greedy bases are a weaker form of quasi-greediness, the efficiency of the greedy algorithm for the former kind of bases is the same as the efficiency we would get for the latter in many important situations. To that end, and for the sake of self-reference, we recall that to measure the performance of the greedy algorithm we compare the error $\Vert f-\GG_{m}(f)\Vert$ in the approximation of any $f\in \XX$ by $\GG_{m}(f)$, with the \emph{best $m$-term approximation error}, given by
\[
\sigma_{m}[\XB, \XX](f):=\sigma_{m}(f) =\inf\{ \Vert f -g \Vert \colon g\in \Sigma_{m}\},
\]
where $\Sigma_m$ denotes the collection of all $f$ in $\XX$ which can be expressed as a linear combination of $m$ elements of $\XB$. An upper estimate for the rate $\Vert f- \GG_m(f)\Vert/\sigma_{m}(f)$ is usually called a \emph{Lebesgue-type inequality} for the TGA (see \cite{Temlyakov2015}*{Chapter 2}). Obtaining Lebesgue-type inequalities is tantamount to finding upper bounds for the \emph{Lebesgue constants} of the basis, given for $m\in \NN$ by
\[
\leb_m=\leb_m[\XB,\XX]=\sup\left\{\frac{\Vert f- \GG_{m}(f)\Vert}{\sigma_{m}(f)}\colon f\in \XX\setminus \Sigma_m\right\}.
\]
By definition, the basis $\XB$ is \emph{greedy} \cite{KoTe1999} if and only if
\[
\Ct_g=\Ct_g[\XB,\XX]:=\sup_m \leb_m<\infty.
\]
Garrig\'os et al.\ showed in \cite{GHO2013}*{Theorem 1.1} that if $\XB$ is a quasi-greedy basis in a Banach space $\XX$ then
\[
\leb_{m}\approx \max\{\dem_m, \unc_m\}, \quad m\in \NN,
\]
where $\dem_m$ is the $m$th \emph{democracy parameter} of the basis,
\[
\dem_m=\dem_m[\XB,\XX]=\sup\limits_{|A|=|B|\le m}\frac{\Vert \sum_{n\in A} \xx_n\Vert}{\Vert \sum_{n\in B} \xx_n\Vert}.
\]
These results have been extended to truncation quasi-greedy bases in quasi-Banach spaces in \cite{AAB2021}, which brought awareness to the importance of obtaining estimates for the democracy parameters for bases in special types of spaces. We find forerunners of this technique in the work of Wojtaszczyk, who proved in \cite{Woj2000}*{Theorem 3} that SUCC bases in a Hilbert space are democratic. More recently it has been shown in \cite{AAW2021} that if $\XB$ is a truncation quasi-greedy basis in $\ell_{p}$ or, more generally, of an $\SL_p$-space, $0<p\le 1$, then $\XB$ is democratic with fundamental function
\[
\varphi[\XB,\XX](m)\approx m^{1/p} , \quad m\in\NN.
\]
Combining this result with Corollary~\ref{thm:UBTOco} we get that if $\XB$ is a truncation quasi-greedy basis of an $\SL_p$-space $\XX$ for $p\in(0,1]\cup\{2,\infty\}$, then $\XB$ is democratic with fundamental function of the same order as $(m^{1/p})_{m=1}^\infty$. Hence,
\[
\leb_{m}[\XB, \XX]\approx \unc_{m}[\XB, \XX]\lesssim (\log m)^{1/p},\quad m\in \NN.
\]

We notice that this result does not hold for $p\in(1,2)\cup(2,\infty)$. Indeed, the canonical basis of $(\oplus_{n=1}^\infty \ell_2^n)_{\ell_p}$ is a non-democratic unconditional basis of a Banach space isomorphic to $\ell_p$.



\end{document}